\documentclass[11pt]{article}
\usepackage{amscd,amsfonts,amsmath,amssymb,amsthm,booktabs,cases,cite,color,diagbox,enumitem,fancyhdr,graphicx}
\usepackage{accents,hyperref,indentfirst,leftidx,mathrsfs,multirow}
\usepackage{picins,ragged2e,rotating,slashbox}
\usepackage{verbatim}
\usepackage[margin=1.2in]{geometry}
\usepackage[all]{xy}
\newtheorem{definition}{Definition}[section]
\newtheorem{theorem}{Theorem}[section]
\newtheorem{lemma}{Lemma}[section]

\numberwithin{equation}{section}

\begin{document}

\title{\bf The Gaussian-Minkowski problem for $C$-pseudo-cones\footnote{This paper was supported by the Excellent Graduate Training Program of SNNU (No. LHRCCX23142).}}
\author{Xudong Wang\thanks{E-mail: xdwang@snnu.edu.cn. School of Mathematics and Statistics, Shaanxi Normal University, Xi'an, 710119, China.}, Tingting Xiang\thanks{E-mail: xiangtt@snnu.edu.cn. School of Mathematics and Statistics, Shaanxi Normal University, Xi'an, 710119, China.}}

\date{}
\maketitle

\begin{abstract}

The Gaussian surface area measures for $C$-pseudo-cones are studied in this paper. Using the variational arguments and the approximation methods of Schneider, we obtain the existence of solutions to the Gaussian-Minkowski problem for $C$-pseudo-cones with small co-volume.



\end{abstract}

\section{Introduction}

The classical Minkowski problem asks for some convex body in the $n$-dimensional Euclidean space $\mathbb{R}^n$ such that its surface area measure is a given finite Borel measure $\mu$ on the unit sphere $\mathbb{S}^{n-1}$. This was solved by Minkowski \cite{Minkowski-konvexenPolyeder,Minkowski-Volumen} in the case of discrete data and continuous density. Aleksandrov \cite{Aleksandrov-surfaceareameasure} and Fenchel-Jessen \cite{FenchelJessen} solved the general case of measure, respectively. Given a smooth function $f$ on $\mathbb{S}^{n-1}$, then the smooth version of the Minkowski problem is a prescribed Gauss curvature problem, namely the following Monge-Amp\`{e}re equation on $\mathbb{S}^{n-1}$:
\begin{equation*}
\det(\nabla^2h+hI)=f,
\end{equation*}
where $\nabla^2$ is the covariant Hessian operator on $\mathbb{S}^{n-1}$ and $I$ is the identity matrix. On this regularity of the Minkowski problem, one can refer to Lewy \cite{Lewy-existenceRiemannianmetric,Lewy-differentialgeometricinlarge}, Nirenberg \cite{Nirenberg}, Pogorelov \cite{Pogorelovbook}, Cheng and Yau \cite{ChengYau-RegularityMinkowskiProblem}, Caffarelli \cite{Caffarelli-InteriorW2pEstimates,Caffarelli-viscositysolutions} and so on.

Recent years, there appear many Minkowski type problems. Such as, $L_p$ Minkowski problem \cite{Lutwak-Lp1}, Orlicz Minkowski problem \cite{Haberl-Lutwak-Yang-Zhang-Orlicz_Minkowski_problem}, dual Minkowski problem \cite{Huang-Lutwak-Yang-Zhang-dual_curvature_measures}, Gauss image problem \cite{Boroczky-Lutwak-Yang-Zhang-Zhao-The_Gauss_Image_Problem}, Chord Minkowski problem \cite{Lutwak-Xi-Yang-Zhang-Chord_measures}, Gaussian-Minkowski problem \cite{Huang-Xi-Zhao-The_Gaussian_Minkowski_problem}, Minkowski problem for electrostatic capacity \cite{Jerison-A_Minkowski_problem_capacity}, torsional rigidity \cite{Colesanti-Fimiani-The_Minkowski_problem_for_torsional_rigidity}, harmonic measure \cite{Jerison-Prescribing_harmonic_measure}, and their various deformations and combinations.

In this paper, we consider the Minkowski type problem for unbounded closed convex sets in the Gaussian probability space. Given a pointed closed convex cone $C$, Schneider \cite{Schneider-A_Brunn_Minkowski_theory,Schneider-Minkowski_type_theorems} firstly studied the Minkowski problem for the surface area measures and the cone-volume measures of unbounded closed convex sets in $C$. Later, many Minkowski type problems for unbounded closed convex sets in Lebesgue measure space \cite{Yang-Ye-Zhu-On_the_Lp,Li-Ye-Zhu-The_dual_Minkowski_problem,Chen-Tu-Regularities,Ai-Yang-Ye-The_L_p_dual_Minkowski_problem,Semenov-Zhao,Zhang-The_Minkowski_problem,Schneider-Gauss_image_problem,Schneider-Pseudo_cones} were also solved. In \cite{Schneider-A_weighted_Minkowski_theorem,Schneider-Weighted_cone_volume_measures_of_pseudo_cones}, Schneider studied the weighted Minkowski problem for $C$-pseudo-cones (see Section 2 for the definition), where the weighted functions have homogeneity and the weighted surface area measures for all $C$-pseudo-cones are finite. Following the same finiteness idea, the Minkowski type problems for $C$-pseudo-cones in the Gaussian probability space are attractive.

The setting of this paper is the $n$-dimensional Euclidean sapce $\mathbb{R}^n$ with the standard inner product $\langle\cdot,\cdot\rangle$ and norm $|\cdot|$. We denote by $\mathbb{S}^{n-1}$ the unit sphere in $\mathbb{R}^n$. We also denote by $\text{{\em int}}$ and $\partial$ the interior operator and the boundary operator, respectively. We fix a $n$-dimensional closed convex cone $C$ in $\mathbb{R}^n$ and suppose that $C$ is pointed, i.e., $C\cap(-C)=\{o\}$. The polar cone of $C$ is defined by
\begin{equation*}
C^\circ=\{x\in\mathbb{R}^n\,|\,\langle x,y\rangle\leqslant0,\ \forall\ y\in C\},
\end{equation*}
and we denote $\Omega_C=\text{{\em int}}\,C\cap\mathbb{S}^{n-1}$, $\Omega_{C^\circ}=\text{{\em int}}\,C^\circ\cap\mathbb{S}^{n-1}$.

Let $E$ be a $C$-pseudo-cone. The Gaussian surface area measure of $E$ is defined by
\begin{equation*}
S_{\gamma_n}(E,\omega)=\frac{1}{(2\pi)^\frac{n}{2}}\int_{\boldsymbol{\nu}^*_E(\omega)}e^{-\frac{|x|^2}{2}}\,d\mathcal{H}^{n-1}(x)
\end{equation*}
for any Borel set $\omega\subset\Omega_{C^\circ}$. Here $\boldsymbol{\nu}^*_E$ represents the inverse Gauss image of $E$ and $\mathcal{H}^k$ represents the $k$-dimensional Hausdorff measure in $\mathbb{R}^n$. In Lemma \ref{Lemma-finiteness-Gauss-surface-area-measure}, we verify the finiteness of the Gaussian surface area measures for all $C$-pseudo-cones. Moreover, the Gaussian co-volume of $E$ is defined by
\begin{equation*}
\overline{\gamma}_n(E)=\gamma_n(C\setminus E)=\frac{1}{(2\pi)^\frac{n}{2}}\int_{C\setminus E}e^{-\frac{|x|^2}{2}}\,d\mathcal{H}^n(x),
\end{equation*}
where $\gamma_n$ is the Gaussian probability measure in $\mathbb{R}^n$. Naturally, the Gaussian-Minkowski problem for $C$-pseudo-cones can be stated as the following:
\vskip 0.3cm
{\em Given a nonzero finite Borel measure $\mu$ on $\Omega_{C^\circ}$, is there a $C$-pseudo-cone $E$ such that the Gaussian surface area measure of $E$ is $\mu$?}
\vskip 0.3cm

In the case of convex bodies, the Gaussian-Minkowski problem has been studied by Huang, Xi and Zhao \cite{Huang-Xi-Zhao-The_Gaussian_Minkowski_problem}. As its unbounded analogue, the existence of some solutions to the Gaussian-Minkowski problem for $C$-pseudo-cones is confirmed. Here we adopt the methods of Schneider \cite{Schneider-A_Brunn_Minkowski_theory,Schneider-A_weighted_Minkowski_theorem} and some technologies in \cite{Huang-Xi-Zhao-The_Gaussian_Minkowski_problem}. By establishing some uniform estimates under the restriction of small co-volume, we obtain the following result:
\begin{theorem}\label{Main-Theorem}
Let $\mu$ be a nonzero Borel measure on $\Omega_{C^\circ}$, then there exists a $C$-pseudo-cone $E$ with $\overline{\gamma}_n(E)\leqslant\frac{1}{2}\gamma_n(C)$ such that
\begin{equation*}
S_{\gamma_n}(E,\cdot)=\mu
\end{equation*}
if and only if $\mu$ is finite.
\end{theorem}

The organization of this paper is as follows. In Section 2, we introduce some tools regarding $C$-pseudo-cones. In Section 3, some properties of the Gaussian surface area measures for $C$-pseudo-cones are studied. Thanks to the uniform estimates in Section 4, the existence of some solutions on compact sets is established by the variational arguments. Finally, the proof of Theorem \ref{Main-Theorem} is given in Section 5.

\section{Preliminaries}

We denote by $B^n$ the unit ball in $\mathbb{R}^n$, and denote by $o$ the origin of $\mathbb{R}^n$. If we choose $\mathfrak{u}\in\Omega_C\cap(-\Omega_{C^\circ})$, then $\langle\mathfrak{u},x\rangle>0$ for any $x\in C\setminus\{o\}$, see \cite{Schneider-Pseudo_cones}. We set $C(t)=\{x\in C\,|\,\langle x,\mathfrak{u}\rangle=t\}$ and $C^-(t)=\{x\in C\,|\,\langle x,\mathfrak{u}\rangle\leqslant t\}$ for $t>0$.

Let $o\notin E\subset C$ be a nonempty closed convex set, then $E$ is called a pseudo-cone if
\begin{equation*}
\lambda x\in E\ \,\text{for any}\ \lambda\geqslant1\ \text{and}\ x\in\,E.
\end{equation*}
Moreover, $E$ is called a $C$-pseudo-cone if ${\rm rec}\,E=C$, where ${\rm rec}\,E$ is the recession cone of $E$ defined by
\begin{equation*}
{\rm rec}\,E=\{x\in\mathbb{R}^n\,|\,E+x\subset E\}.
\end{equation*}
We denote by $\mathcal{PC}(C)$ the set of all $C$-pseudo-cones.

Let $E\in\mathcal{PC}(C)$. If $C\setminus E$ has finite volume, then $E$ is called a $C$-close set and $C\setminus E$ is called a $C$-coconvex set. In particular, a $C$-close set $E$ is called a $C$-full set if $C\setminus E$ is bounded. The support function of $E$ is defined by
\begin{equation*}
h_E(v)=\sup\{\langle x,v\rangle\,|\,x\in E\},\,v\in\Omega_{C^{\circ}}.
\end{equation*}
Note that $-\infty<h_E(v)<0$ for any $v\in\Omega_{C^{\circ}}$, one can define $\overline{h}_E=-h_E$, so that $\overline{h}_E$ is a positive continuous function on $\Omega_{C^{\circ}}$. The radial function of $E$ is defined by
\begin{equation*}
\rho_E(u)=\inf\{r>0\,|\,ru\in E\},\,u\in\Omega_C.
\end{equation*}

Let $\omega\subset\Omega_{C^\circ}$ be a nonempty compact set and $K\in\mathcal{PC}(C)$. $K$ is called a $C$-determined set by $\omega$ if
\begin{equation*}
K=C\cap\bigcap_{v\in\omega}H^-(v,h_K(v)).
\end{equation*}
We denote by $\mathcal{K}(C,\omega)$ the class of all $C$-determined sets by $\omega$. Note that a $C$-determined set by a compact set is also a $C$-full set. Moreover, one can define a $C$-determined set $K^{(\omega)}$ of a $C$-pseudo-cone $K$ by
\begin{equation*}
K^{(\omega)}=C\cap\bigcap_{v\in\omega}H^-(v,h_K(v)).
\end{equation*}

Denote by $\widetilde{\partial}E=\partial E\cap\text{{\em int}}\,C$. We use $\mathfrak{B}(X)$ to represent the class of all Borel sets in $X$. For $E\in\mathcal{PC}(C)$ and $\beta\in\mathfrak{B}(\widetilde{\partial}E)$, the Gauss image of $E$ is defined by
\begin{equation*}
\boldsymbol{\nu}_E(\beta)=\{v\in\Omega_{C^\circ}\,|\,\mbox{there exists}\ x\in\beta\ \mbox{such that}\ \langle x,v\rangle=h_E(v)\},
\end{equation*}
and for $\omega\in\mathfrak{B}(\Omega_{C^\circ})$, the inverse Gauss image of $E$ is defined by
\begin{equation*}
\boldsymbol{\nu}^*_E(\omega)=\{x\in\widetilde{\partial}E\,|\,\mbox{there exists}\ v\in\omega\ \mbox{such that}\ \langle x,v\rangle=h_E(v)\}.
\end{equation*}

Let $\{E_i\,|\,i\in\mathbb{N}_0\}\subset\mathcal{PC}(C)$, we say that $E_i$ converges to $E_0$ as $i\rightarrow+\infty$ in the sense of Schneider if there exists $t_0>0$ such that $E_i\cap C^-(t_0)\neq\emptyset$ for all $i\geqslant0$ and
\begin{equation*}
E_i\cap C^-(t)\ \mbox{converges to}\ E_0\cap C^-(t)\ \mbox{as}\ i\rightarrow+\infty
\end{equation*}
with respect to the Hausdorff metric for all $t\geqslant t_0$. We write
\begin{equation*}
b(E)=d(o,E)=\inf\{|x|\,|\,x\in E\},\ \text{for}\ E\in\mathcal{PC}(C).
\end{equation*}

\begin{lemma}[Schneider selection theorem, see \cite{Schneider-Pseudo_cones}]\label{Lemma-Schneider-selection-theorem}
Let $\{E_i\,|\,i\in\mathbb{N}_0\}\subset\mathcal{PC}(C)$. If there exist two positive constants $\lambda$ and $\Lambda$ such that
\begin{equation*}
\lambda<b(E_i)<\Lambda,\ \forall\ i\geqslant1,
\end{equation*}
then $\{E_i\}_{i=1}^{+\infty}$ has a subsequence that converges to a $C$-pseudo-cone (in the sense of Schneider).
\end{lemma}

\section{Gaussian surface area measures for $C$-pseudo-cones}

In this section, we introduce some definitions and related properties of the Gaussian surface area measures for $C$-pseudo-cones.
\begin{definition}
Let $E\in\mathcal{PC}(C)$. The Gaussian surface area measure of $E$ is defined by
\begin{equation*}
S_{\gamma_n}(E,\omega)=\frac{1}{(2\pi)^\frac{n}{2}}\int_{\boldsymbol{\nu}^*_E(\omega)}e^{-\frac{|x|^2}{2}}\,d\mathcal{H}^{n-1}(x)
\end{equation*}
for any $\omega\in\mathfrak{B}(\Omega_{C^\circ})$. And the Gaussian co-volume of $E$ is defined by
\begin{equation*}
\overline{\gamma}_n(E)=\frac{1}{(2\pi)^\frac{n}{2}}\int_{C\setminus E}e^{-\frac{|x|^2}{2}}\,d\mathcal{H}^n(x).
\end{equation*}
\end{definition}

Let $E\in\mathcal{PC}(C)$, $x_0\in\text{int}\,E$ and $t_0=\langle x_0,\mathfrak{u}\rangle$. Define $t_i=t_0+i$ for $i\in\mathbb{N}$, and the following sets
\begin{equation*}
\left\{
\begin{aligned}
E_i&=\{x\in E\,|\,t_i\leqslant\langle x,\mathfrak{u}\rangle\leqslant t_{i+1}\},\ \partial^*E_i=\{x\in\partial E_i\,|\,t_i<\langle x,\mathfrak{u}\rangle<t_{i+1}\},\\
\overline{E_i}&=E\cap C(t_{i+1})+\{\lambda\mathfrak{u}\,|\,-1\leqslant\lambda\leqslant0\},\ \partial^*\overline{E_i}= \{x\in\partial\overline{E_i}\,|\,t_i<\langle x,\mathfrak{u}\rangle<t_{i+1}\},\\
\underline{E_i}&=E\cap C(t_i)+\{\lambda\mathfrak{u}\,|\,0\leqslant\lambda\leqslant1\},\ \partial^*\underline{E_i}= \{x\in\partial\underline{E_i}\,|\,t_i<\langle x,\mathfrak{u}\rangle<t_{i+1}\},
\end{aligned}
\right.
\end{equation*}
where the upper cylinders $\overline{E_i}$, lower cylinders $\underline{E_i}$, and convex bodies $E_i$ satisfy
\begin{equation*}
\underline{E_i}\subset E_i\subset\overline{E_i}.
\end{equation*}
Thanks to this partition, we can prove the finiteness of the Gaussian surface area measures for $C$-pseudo-cones.
\begin{lemma}\label{Lemma-finiteness-Gauss-surface-area-measure}
Let $E\in\mathcal{PC}(C)$, then $S_{\gamma_n}(E,\Omega_{C^\circ})<+\infty$.
\end{lemma}
\begin{proof}
Without loss of generality, we assume that $E\cap\partial C=\emptyset$, so $\widetilde{\partial}E=\partial E$. Then, the Gaussian surface area of $E$ is divided into
\begin{equation*}
\begin{aligned}
S_{\gamma_n}(E,\Omega_{C^\circ})&=\frac{1}{(2\pi)^\frac{n}{2}}\int_{\widetilde{\partial}E=\partial E}e^{-\frac{|x|^2}{2}}
\,d\mathcal{H}^{n-1}(x)\\
&=\frac{1}{(2\pi)^\frac{n}{2}}\int_{\partial E\cap C^-(t_0)}e^{-\frac{|x|^2}{2}}\,d\mathcal{H}^{n-1}(x)
+\frac{1}{(2\pi)^\frac{n}{2}}\sum_{i=0}^{+\infty}\int_{\partial^*E_i}e^{-\frac{|x|^2}{2}}\,d\mathcal{H}^{n-1}(x)\\
&\triangleq I+J.
\end{aligned}
\end{equation*}

On $\partial^*E_i$, there holds $t_i<\langle x,\mathfrak{u}\rangle\leqslant|x|$. Thus, we have
\begin{equation*}
e^{-\frac{|x|^2}{2}}\leqslant e^{-\frac{(t_0+i)^2}{2}}.
\end{equation*}
Moreover, by an estimate in \cite{Schneider-A_weighted_Minkowski_theorem} (see also Lemma 10 in \cite{Wang-Xu-Zhou-Zhu-Asymptotic_theory} for its proof), there exists a constant $C_1$ which is independent of $i\in\mathbb{N}$, such that
\begin{equation*}
\mathcal{H}^{n-1}((\overline{E_i}\setminus E_i)\cap C(t_i))\leqslant C_1\mathcal{H}^{n-2}(\partial E\cap C(t_i)).
\end{equation*}
Then, there holds
\begin{equation*}
\begin{aligned}
J&=\frac{1}{(2\pi)^\frac{n}{2}}\sum_{i=0}^{+\infty}\int_{\partial^*E_i}e^{-\frac{|x|^2}{2}}\,d\mathcal{H}^{n-1}(x)\\
&\leqslant\frac{1}{(2\pi)^\frac{n}{2}}\sum_{i=0}^{+\infty}e^{-\frac{(t_0+i)^2}{2}}\mathcal{H}^{n-1}(\partial^*E_i)\\
&\leqslant\frac{1}{(2\pi)^\frac{n}{2}}\sum_{i=0}^{+\infty}e^{-\frac{(t_0+i)^2}{2}}\big(\mathcal{H}^{n-2}(\partial E\cap C(t_{i+1}))
+C_1\mathcal{H}^{n-2}(\partial E\cap C(t_i))\big)\\
&\leqslant\frac{1+C_1}{(2\pi)^\frac{n}{2}}\sum_{i=0}^{+\infty}e^{-\frac{(t_0+i)^2}{2}}\mathcal{H}^{n-2}(\partial C\cap C(t_{i+1}))\\
&=\frac{1+C_1}{(2\pi)^\frac{n}{2}}\mathcal{H}^{n-2}(\partial C\cap C(1))\sum_{i=0}^{+\infty}e^{-\frac{(t_0+i)^2}{2}}(t_0+i+1)^{n-2},
\end{aligned}
\end{equation*}
where the second line used the monotonicity of surface area of convex sets. Let $u_i=e^{-\frac{(t_0+i)^2}{2}}(t_0+i+1)^{n-2}$, then
\begin{equation*}
\lim_{i\rightarrow+\infty}\frac{u_{i+1}}{u_i}=0.
\end{equation*}
Then, $\sum_{i=0}^{+\infty}u_i$ converges. This shows $S_{\gamma_n}(E,\Omega_{C^\circ})=I+J<+\infty$.
\end{proof}

Let $\omega\subset\Omega_{C^\circ}$ be a compact set. Denote by $\mathcal{C}(\omega)$ the class of all continuous functions on $\omega$ and denote by $\mathcal{C}^+(\omega)\subset\mathcal{C}(\omega)$ the class of all positive continuous functions on $\omega$. The Wulff shape of $f\in\mathcal{C}^+(\omega)$ is defined by
\begin{equation*}
[f]=C\cap\bigcap_{v\in\omega}\{x\in\mathbb{R}^n\,|\,\langle x,v\rangle\leqslant-f(v)\},
\end{equation*}
which is a $C$-determined set by $\omega$, and so it is also a $C$-pseudo-cone.
\begin{lemma}\label{Lemma-Variational-formula}
Suppose that $\omega\subset\Omega_{C^\circ}$ is a nonempty and compact set. Let $K\in\mathcal{K}(C,\omega)$ and $f\in\mathcal{C}(\omega)$, then
\begin{equation*}
\delta\overline{\gamma}_n(K)(f)\triangleq\lim_{t\rightarrow0}\frac{\overline{\gamma}_n([\overline{h}_K|_\omega+tf])
-\overline{\gamma}_n(K)}{t}=\int_{\omega}f(v)\,dS_{\gamma_n}(K,v).
\end{equation*}
\end{lemma}
\begin{proof}
Let $h_t=\overline{h}_K|_\omega+tf$. Using polar coordinates, we have
\begin{equation*}
\overline{\gamma}_n([h_t])=\frac{1}{(2\pi)^\frac{n}{2}}\int_{C\setminus[h_t]}e^{-\frac{|x|^2}{2}}\,d\mathcal{H}^n(x)
=\frac{1}{(2\pi)^\frac{n}{2}}\int_{\Omega_C}\int_0^{\rho_{[h_t]}(u)}e^{-\frac{r^2}{2}}r^{n-1}\,drdu.
\end{equation*}
Now, we define the function
\begin{equation*}
F_u(t)=\int_0^{\rho_{[h_t]}(u)}e^{-\frac{r^2}{2}}r^{n-1}\,dr.
\end{equation*}
By Lemma 10 in \cite{Schneider-A_weighted_Minkowski_theorem} and the mean value theorem of integrals, we have
\begin{equation*}
\begin{aligned}
\frac{F_u(t)-F_u(0)}{t}&=\frac{\rho_{[h_t]}(u)-\rho_K(u)}{t}\cdot\frac{1}{\rho_{[h_t]}(u)-\rho_K(u)}\int_{\rho_K(u)}^{\rho_{[h_t]}(u)}
e^{-\frac{r^2}{2}}r^{n-1}\,dr\\
&\rightarrow\frac{f(\alpha_K(u))}{\overline{h}_K(\alpha_K(u))}e^{-\frac{1}{2}\rho^2_K(u)}\rho^n_K(u),\ \mbox{for almost all}\ u\in\Omega_C,
\end{aligned}
\end{equation*}
and there is a constant $M$ (depending on $K$ and $f$) such that
\begin{equation*}
\Big|\frac{F_u(t)-F_u(0)}{t}\Big|=\Big|\frac{\rho_{[h_t]}(u)-\rho_K(u)}{t}\Big|\cdot e^{-\frac{\xi^2}{2}}\xi^{n-1}\leqslant M(\rho_K(u)+M)^{n-1}
\end{equation*}
for $t$ close to $0$. Since $K\in\mathcal{K}(C,\omega)$, there is another constant $M'$ depending only on $K$, such that $\rho_K(u)\leqslant M'$ for all $u\in\Omega_C$. Thus, for $t$ close to $0$, we have
\begin{equation*}
\Big|\frac{F_u(t)-F_u(0)}{t}\Big|\leqslant M''\triangleq M(M'+M)^{n-1}.
\end{equation*}

Finally, using the dominated convergence theorem, the Co-area formula and the push-forward measure, there holds
\begin{equation*}
\begin{aligned}
\lim_{t\rightarrow0}\frac{\overline{\gamma}_n([h_t])-\overline{\gamma}_n(K)}{t}&=\frac{1}{(2\pi)^\frac{n}{2}}\lim_{t\rightarrow0}
\int_{\Omega_C}\frac{F_u(t)-F_u(0)}{t}\,du\\
&=\frac{1}{(2\pi)^\frac{n}{2}}\int_{\Omega_C}\frac{f(\alpha_K(u))}{\overline{h}_K(\alpha_K(u))}e^{-\frac{1}{2}\rho^2_K(u)}\rho^n_K(u)\,du\\
&=\frac{1}{(2\pi)^\frac{n}{2}}\int_{\widetilde{\partial}K}f(\nu_K(x))e^{-\frac{|x|^2}{2}}\,d\mathcal{H}^{n-1}(x)\\
&=\int_{\omega}f(v)\,dS_{\gamma_n}(K,v).
\end{aligned}
\end{equation*}
In last line, we noted that $S_{\gamma_n}(K,\cdot)$ is concentrated on $\omega$.
\end{proof}

The following result shows that the Gaussian surface area measures for $C$-determined sets are weakly continuous.
\begin{lemma}\label{Lemma-weak-continuity-Gauss-surface-area-measure}
Let $\omega\subset\Omega_{C^\circ}$ be a nonempty compact set. If $K_i\in\mathcal{K}(C,\omega)$ converges to $K_0\in\mathcal{K}(C,\omega)$ as $i\rightarrow+\infty$, then $S_{\gamma_n}(K_i,\cdot)$ converges to $S_{\gamma_n}(K_0,\cdot)$ weakly.
\end{lemma}
\begin{proof}
For any bounded continuous function $f:\Omega_{C^\circ}\rightarrow\mathbb{R}$, by the Co-area formula and the push-forward measure we have
\begin{equation*}
\begin{aligned}
\int_{\Omega_{C^\circ}}f(v)\,dS_{\gamma_n}(K_i,v)=\frac{1}{(2\pi)^\frac{n}{2}}\int_{\Omega_C}f(\alpha_{K_i}(u))
\frac{\rho^n_{K_i}(u)}{\overline{h}_{K_i}(\alpha_{K_i}(u))}e^{-\frac{1}{2}\rho^2_{K_i}(u)}\,du.
\end{aligned}
\end{equation*}
Since $K_i$ converges to $K_0\in\mathcal{K}(C,\omega)$, $\rho_{K_i}$ is uniformly bounded on $\Omega_C$. Moreover, all normal vectors of $K_i$ are concentrated on $\omega$, so $\overline{h}_{K_i}(\alpha_{K_i}(u))$ is bounded away from $0$. Thus, $f(\alpha_{K_i}(u))\frac{\rho^n_{K_i}(u)}{\overline{h}_{K_i}(\alpha_{K_i}(u))}e^{-\frac{1}{2}\rho^2_{K_i}(u)}$ is uniformly bounded on $\Omega_C$. Then, the dominated convergence theorem shows
\begin{equation*}
\begin{aligned}
&\lim_{i\rightarrow+\infty}\int_{\Omega_C}f(\alpha_{K_i}(u))\frac{\rho^n_{K_i}(u)}{\overline{h}_{K_i}(\alpha_{K_i}(u))}
e^{-\frac{1}{2}\rho^2_{K_i}(u)}\,du\\
=&\int_{\Omega_C}\lim_{i\rightarrow+\infty}f(\alpha_{K_i}(u))\frac{\rho^n_{K_i}(u)}{\overline{h}_{K_i}(\alpha_{K_i}(u))}
e^{-\frac{1}{2}\rho^2_{K_i}(u)}\,du\\
=&\int_{\Omega_C}f(\alpha_{K_0}(u))\frac{\rho^n_{K_0}(u)}{\overline{h}_{K_0}(\alpha_{K_0}(u))}e^{-\frac{1}{2}\rho^2_{K_0}(u)}\,du.
\end{aligned}
\end{equation*}
By the Co-area formula and the push-forward measure again, we have
\begin{equation*}
\lim_{i\rightarrow+\infty}\int_{\Omega_{C^\circ}}f(v)\,dS_{\gamma_n}(K_i,v)=\int_{\Omega_{C^\circ}}f(v)\,dS_{\gamma_n}(K_0,v),
\end{equation*}
which implies the weak convergence $S_{\gamma_n}(K_i,\cdot)\rightarrow S_{\gamma_n}(K_0,\cdot)$.
\end{proof}

\begin{lemma}\label{Lemma-continuity-Gauss-co-volume}
Let $\{E_i\,|\,i\in\mathbb{N}_0\}\subset\mathcal{PC}(C)$ and $E_i\rightarrow E_0$ as $i\rightarrow+\infty$, then
\begin{equation*}
\lim_{i\rightarrow+\infty}\overline{\gamma}_n(E_i)=\overline{\gamma}_n(E_0).
\end{equation*}
\end{lemma}
\begin{proof}
Note that $0<\overline{\gamma}_n(E)<\gamma_n(C)<\frac{1}{2}$ for any $E\in\mathcal{PC}(C)$. Choose $z\in \text{int}\,E_0$. For any $\varepsilon>0$, there exists $t>0$ such that
\begin{equation*}
\gamma_n\big((C\setminus(z+C))\setminus C^-(t)\big)<\frac{1}{4}\varepsilon,
\end{equation*}
where $z+C$ is also a $C$-pseudo-cone.

As Lemma 7 in \cite{Schneider-A_weighted_Minkowski_theorem}, there is a constant $a$ depending only on $C^-(t)$ such that $\mathcal{H}^n(A\setminus B)\leqslant\delta a$ for convex bodies $A,B\subset C^-(t)$ with $d_H(A,B)\leqslant\delta$. Since $E_i\rightarrow E_0$ as $i\rightarrow+\infty$, there is $N>0$ such that
\begin{equation*}
d_H(E_i\cap C^-(t),E_0\cap C^-(t))\leqslant\frac{1}{4a}\varepsilon\ \mbox{and}\ z\in E_i
\end{equation*}
for $i>N$. Thus, we have
\begin{equation*}
\mathcal{H}^n\big((E_i\setminus E_0)\cap C^-(t)\big)\leqslant\frac{1}{4}\varepsilon\ \ \mbox{and}\ \
\mathcal{H}^n\big((E_0\setminus E_i)\cap C^-(t)\big)\leqslant\frac{1}{4}\varepsilon\ \ \mbox{for all}\ i>N.
\end{equation*}
Therefore, for above $\varepsilon>0$, we have
\begin{equation*}
\begin{aligned}
\big|\overline{\gamma}_n(E_i)-\overline{\gamma}_n(E_0)\big|&\leqslant\gamma_n((E_0\setminus E_i)\cap C^-(t))+\gamma_n((E_i\setminus E_0)\cap C^-(t))\\
&\quad+\gamma_n((C\setminus E_i)\setminus C^-(t))+\gamma_n((C\setminus E_0)\setminus C^-(t))\\
&\leqslant\mathcal{H}^n((E_0\setminus E_i)\cap C^-(t)+\mathcal{H}^n((E_i\setminus E_0)\cap C^-(t))\\
&\quad+\gamma_n((C\setminus(z+C))\setminus C^-(t))+\gamma_n((C\setminus(z+C))\setminus C^-(t))\\
&\leqslant\frac{1}{4}\varepsilon+\frac{1}{4}\varepsilon+\frac{1}{4}\varepsilon+\frac{1}{4}\varepsilon=\varepsilon,\ \mbox{for}\ i>N.
\end{aligned}
\end{equation*}
That is, $\lim_{i\rightarrow+\infty}\overline{\gamma}_n(E_i)=\overline{\gamma}_n(E_0)$.
\end{proof}

\section{The variational results on compact sets}

Let $\omega\subset\Omega_{C^\circ}$ be a nonempty compact set and $\mu$ be a finite measure on $\omega$. Given $\alpha>\frac{1}{n}$, we consider the following extremal problem:
\begin{equation*}
\sup\{\mathcal{F}(f)\,|\,f\in\mathcal{C}^+(\omega)\ \mbox{and}\ \overline{\gamma}_n([f])\leqslant\frac{1}{2}\beta\},
\end{equation*}
where
\begin{equation*}
\mathcal{F}(f)=\int_\omega f\,d\mu-\frac{1}{\alpha}\overline{\gamma}_n([f])^\alpha,
\end{equation*}
and
\begin{equation*}
\beta=\frac{\mathcal{H}^{n-1}(\mathbb{S}^{n-1}\cap C)}{\mathcal{H}^{n-1}(\mathbb{S}^{n-1})}=\gamma_n(C).
\end{equation*}

\begin{lemma}\label{Lemma-delta-function}
There holds $\mathcal{F}(\delta)>0$ for a sufficiently small positive constant function $\delta$ in $\mathcal{C}^+(\omega)$.
\end{lemma}
\begin{proof}
Let $1$ be the unit constant function on $\omega$. Since $[1]$ is a $C$-full set, there is a constant $s>0$ such that
\begin{equation*}
\rho_{[1]}(u)\leqslant s,\,\forall\ u\in\Omega_C,\ \mbox{i.e.}\ \ C\setminus[1]\subset sB^n\cap C.
\end{equation*}
Then, by the homogeneity of the Wulff shape, one has $C\setminus[\delta]\subset[(s\delta)B^n]\cap C$ for a positive constant function $\delta$ in $\mathcal{C}^+(\omega)$. Similar to Lemma 3.7 in \cite{Huang-Xi-Zhao-The_Gaussian_Minkowski_problem} but with a little difference, we have
\begin{equation*}
\begin{aligned}
\mathcal{F}(\delta)&=\int_\omega\delta\,d\mu-\frac{1}{\alpha}\overline{\gamma}_n([\delta])^\alpha\\
&>\delta\,\mu(\omega)-\frac{1}{\alpha}\gamma_n([(s\delta)B^n]\cap C)^\alpha\\
&=\delta\,\mu(\omega)-\frac{1}{\alpha}\bigg(\frac{\mathcal{H}^{n-1}(\mathbb{S}^{n-1}\cap C)}{(2\pi)^\frac{n}{2}}
\int_0^{s\delta}e^{-\frac{t^2}{2}}t^{n-1}\,dt\bigg)^\alpha\\
&=\delta\bigg[\mu(\omega)-\frac{1}{\alpha}\bigg(\frac{\mathcal{H}^{n-1}(\mathbb{S}^{n-1}\cap C)}{(2\pi)^\frac{n}{2}}\bigg)^\alpha
\bigg(\delta^{-\frac{1}{\alpha}}\int_0^{s\delta}e^{-\frac{t^2}{2}}t^{n-1}\,dt\bigg)^\alpha\bigg].
\end{aligned}
\end{equation*}
For $\alpha>\frac{1}{n}$, there holds
\begin{equation*}
\lim_{\delta\rightarrow0^+}\delta^{-\frac{1}{\alpha}}\int_0^{s\delta}e^{-\frac{t^2}{2}}t^{n-1}\,dt
=\alpha s^n\lim_{\delta\rightarrow0^+}\delta^{n-\frac{1}{\alpha}}=0.
\end{equation*}
Therefore, $\mathcal{F}(\delta)>0$ for a sufficiently small positive constant function $\delta$ in $\mathcal{C}^+(\omega)$.
\end{proof}

\begin{lemma}[Upper bound estimate]\label{Lemma-Upper-bound-estimate}
There exists a constant $\Lambda$, such that
\begin{equation*}
b(E)\leqslant\Lambda
\end{equation*}
for all $E\in\mathcal{PC}(C)$ satisfying $\overline{\gamma}_n(E)\leqslant\frac{1}{2}\beta$.
\end{lemma}
\begin{proof}
Suppose that $\{E_i\,|\,i\in\mathbb{N}_0\}\subset\mathcal{PC}(C)$ satisfies
\begin{equation*}
\overline{\gamma}_n(E_i)\leqslant\frac{1}{2}\beta\ \ \mbox{and}\ \lim_{i\rightarrow+\infty}b(E_i)=+\infty.
\end{equation*}
\newpage
Since $b(E_i)B^n\cap C\subset C\setminus E_i$, we have
\begin{equation*}
\gamma_n(C)\geqslant\overline{\gamma}_n(E_i)=\gamma_n(C\setminus E_i)\geqslant\gamma_n(b(E_i)B^n\cap C).
\end{equation*}
Letting $i\rightarrow+\infty$, then one has
\begin{equation*}
\lim_{i\rightarrow+\infty}\overline{\gamma}_n(E_i)=\gamma_n(C)=\beta,
\end{equation*}
which contradicts to $\overline{\gamma}_n(E_i)\leqslant\frac{1}{2}\beta$.
\end{proof}

\begin{theorem}\label{Theorem-solution-compact-set}
Suppose that $\omega\subset\Omega_{C^\circ}$ is a nonempty compact set and $n\alpha>1$. Let $\mu$ be a nonzero Borel measure on $\omega$, then there is $K\in\mathcal{K}(C,\omega)$ with $\overline{\gamma}_n(K)\leqslant\frac{1}{2}\beta$ such that
\begin{equation*}
\overline{\gamma}_n(K)^{\alpha-1}S_{\gamma_n}(K,\cdot)=\mu
\end{equation*}
if and only if $\mu$ is finite.
\end{theorem}
\begin{proof}
$(i)$ Let $\mu$ be nonzero and finite. For any $f\in\mathcal{C}^+(\omega)$ with $\overline{\gamma}_n([f])\leqslant\frac{1}{2}\beta$, by Lemma \ref{Lemma-Upper-bound-estimate} we have $b([f])\leqslant\Lambda$ for a constant $\Lambda$. According to the definition of the Wulff shape and Lemma 1 in \cite{Schneider-Minkowski_type_theorems}, one has
\begin{equation*}
-b([f])\leqslant h_{[f]}\leqslant-f,
\end{equation*}
which shows
\begin{equation*}
\mathcal{F}(f)=\int_\omega f\,d\mu-\frac{1}{\alpha}\overline{\gamma}_n([f])^\alpha
\leqslant b([f])\int_\omega\,d\mu\leqslant\Lambda|\mu|.
\end{equation*}
Combining this with Lemma \ref{Lemma-delta-function}, there holds
\begin{equation*}
0<\sup\{\mathcal{F}(f)\,|\,f\in\mathcal{C}^+(\omega)\ \mbox{and}\ \overline{\gamma}_n([f])\leqslant\frac{1}{2}\beta\}\triangleq A<+\infty.
\end{equation*}

Choose a sequence of $f_i\in\mathcal{C}^+(\omega)$ with $\overline{\gamma}_n([f_i])\leqslant\frac{1}{2}\beta$, such that
\begin{equation*}
\lim_{i\rightarrow+\infty}\mathcal{F}(f_i)=A.
\end{equation*}
There exists a constant $\lambda>0$, such that $b([f_i])\geqslant\lambda$ for all $i\in\mathbb{N}$. Otherwise, there is a subsequence $\{f_{i_k}\}$ of $\{f_i\}$ such that $b([f_{i_k}])\rightarrow0$ as $k\rightarrow+\infty$. From $0<f_{i_k}\leqslant b([f_{i_k}])$, we know $f_{i_k}\rightarrow0$ uniformly on $\omega$ as $k\rightarrow+\infty$. And so $\overline{\gamma}_n([f_{i_k}])\rightarrow0$ as $k\rightarrow+\infty$. Then,
\begin{equation*}
\lim_{k\rightarrow+\infty}\mathcal{F}(f_{i_k})=\lim_{k\rightarrow+\infty}\int_\omega f_{i_k}\,d\mu
-\lim_{k\rightarrow+\infty}\frac{1}{\alpha}\overline{\gamma}_n([f_{i_k}])^\alpha=0.
\end{equation*}
Hence, we have the following contradiction
\begin{equation*}
\lim_{k\rightarrow+\infty}\mathcal{F}(f_{i_k})=\lim_{i\rightarrow+\infty}\mathcal{F}(f_i)=A>0.
\end{equation*}

In a word, there holds $0<\lambda\leqslant b([f_i])\leqslant\Lambda$. By the Schneider selection theorem (Lemma \ref{Lemma-Schneider-selection-theorem}), there is a $C$-pseudo-cone $K$ such that
\begin{equation*}
[f_i]\rightarrow K,\ \mbox{as}\ i\rightarrow+\infty.
\end{equation*}
Here we still denote the subsequence in Lemma \ref{Lemma-Schneider-selection-theorem} by $[f_i]$. For any support hyperplane $H(v,\tau)$ of $[f_i]$, in which $v\in\omega$, we have
\begin{equation*}
H(v,\tau)\cap(\Lambda B^n\cap C)\neq\emptyset.
\end{equation*}
Otherwise, $H(v,\tau)\cap(\Lambda B^n\cap C)=\emptyset$ contradicts to $b([f_i])\leqslant\Lambda$. Note that $\Lambda B^n\cap C$ is a bounded set in $C$, by Lemma 8 in \cite{Schneider-A_Brunn_Minkowski_theory}, there is a constant $T>0$ such that
\begin{equation*}
H(v,\tau)\cap C\subset C^-(T).
\end{equation*}
Thus, we have $\boldsymbol{\nu}^*_{[f_i]}(\omega)=\widetilde{\partial}[f_i]\subset C^-(T)$, i.e., $\rho_{[f_i]}\leqslant T'$ for a constant $T'$ depending on $C$, $T$ and $\mathfrak{u}$. According to the definition of convergence of $C$-pseudo-cones in the sense of Schneider, we know that $K$ is a $C$-full set and $\rho_K\leqslant T'$. Hence, by Lemma 6 in \cite{Schneider-A_Brunn_Minkowski_theory}, we have $K\in\mathcal{K}(C,\omega)$. And by Lemma \ref{Lemma-continuity-Gauss-co-volume}, one has $\overline{\gamma}_n(K)\leqslant\frac{1}{2}\beta$. From
\begin{equation*}
\mathcal{F}(f_i)\leqslant\mathcal{F}(\overline{h}_{[f_i]})\rightarrow\mathcal{F}(\overline{h}_K|_\omega),\ \text{as}\ i\rightarrow+\infty,
\end{equation*}
one has $A\leqslant\mathcal{F}(\overline{h}_K|_\omega)$. Therefore, $\overline{h}_K|_\omega$ is the maximizer of the functional $\mathcal{F}$.

For any $f\in\mathcal{C}(\omega)$ and sufficiently small $t$, by Lemma \ref{Lemma-Variational-formula}, we have
\begin{equation*}
\begin{aligned}
0&=\delta\mathcal{F}(\overline{h}_K|_\omega)(f)\\
&=\int_\omega f\,d\mu-\overline{\gamma}_n(K)^{\alpha-1}\delta\overline{\gamma}_n(K)(f)\\
&=\int_\omega f\,d\mu-\overline{\gamma}_n(K)^{\alpha-1}\int_{\omega}f(v)\,dS_{\gamma_n}(K,v).
\end{aligned}
\end{equation*}
Thus, the Euler-Lagrange equation is
\begin{equation*}
\overline{\gamma}_n(K)^{\alpha-1}\int_{\omega}f(v)\,dS_{\gamma_n}(K,v)=\int_\omega f\,d\mu.
\end{equation*}
By the Riesz representation theorem, one has
\begin{equation*}
\overline{\gamma}_n(K)^{\alpha-1}S_{\gamma_n}(K,\cdot)=\mu.
\end{equation*}
$(ii)$ If there is $K\in\mathcal{K}(C,\omega)$ with $\overline{\gamma}_n(K)\leqslant\frac{1}{2}\beta$ such that $\overline{\gamma}_n(K)^{\alpha-1}S_{\gamma_n}(K,\cdot)=\mu$, then $\mu$ is a nonzero finite Borel measure due to Lemma \ref{Lemma-finiteness-Gauss-surface-area-measure}.
\end{proof}

\section{Proof of Theorem \ref{Main-Theorem}}

Now, we use the approximation methods in \cite{Schneider-Pseudo_cones,Schneider-A_weighted_Minkowski_theorem} to prove the following results, which is Theorem \ref{Main-Theorem} for $\alpha=1$.
\begin{theorem}
Let $\mu$ be a Borel measure on $\Omega_{C^\circ}$ and $\alpha\geqslant1$, then there is a $C$-pseudo-cone $E$ with $\overline{\gamma}_n(E)\leqslant\frac{1}{2}\beta$ such that
\begin{equation*}
\overline{\gamma}_n(E)^{\alpha-1}S_{\gamma_n}(E,\cdot)=\mu
\end{equation*}
if and only if $\mu$ is finite.
\end{theorem}
\begin{proof}
$(i)$ Let $\mu$ be a nonzero finite Borel measure on $\Omega_{C^\circ}$. We choose a sequence of compact sets $\{\omega_i\}_{i=1}^{+\infty}\subset\Omega_{C^\circ}$ such that
\begin{equation*}
\mu(\omega_1)>0,\ \omega_i\subset\text{int}\,\omega_{i+1}\ \mbox{and}\ \bigcup_{i=1}^{+\infty}\omega_i=\Omega_{C^\circ}.
\end{equation*}
Define $\mu_i=\mu\llcorner\omega_i$ (i.e., $\mu_i(\omega)=\mu(\omega\cap\omega_i)$ for any $\omega\subset\Omega_{C^\circ}$), then $\mu_i$ is a nonzero finite Borel measure concentrated on $\omega_i$.

Applying Theorem \ref{Theorem-solution-compact-set} on $\omega_i$, then there exists $K_i\in\mathcal{K}(C,\omega_i)$ such that
\begin{equation*}
\overline{\gamma}_n(K_i)^{\alpha-1}S_{\gamma_n}(K_i,\cdot)=\mu_i\ \ \mbox{and}\ \ \overline{\gamma}_n(K_i)\leqslant\frac{1}{2}\beta,
\end{equation*}
where $\alpha\geqslant1$. Note that
\begin{equation*}
S_{n-1}(K_i,\omega_1)\geqslant S_{\gamma_n}(K_i,\omega_1)=\overline{\gamma}_n(K_i)^{1-\alpha}\mu_i(\omega_1)\geqslant\mu_i(\omega_1)
=\mu(\omega_1)\triangleq s>0,
\end{equation*}
by Lemma 9 in \cite{Schneider-Pseudo_cones} there is a constant $\lambda>0$ depending only on $C$, $\omega_1$ and $s$, such that
\begin{equation*}
b(K_i)\geqslant\lambda.
\end{equation*}
Meanwhile, by Lemma \ref{Lemma-Upper-bound-estimate}, one has
\begin{equation*}
b(K_i)\leqslant\Lambda
\end{equation*}
for a constant $\Lambda$. Hence, by the Schneider selection theorem, we can say that $K_i$ converges to a $C$-pseudo-cone $K$ as $i\rightarrow+\infty$.

Choose an increasing sequence $\{t_k\}_{k=1}^{+\infty}$ with $t_k\rightarrow+\infty$ as $t\rightarrow+\infty$, then
\begin{equation*}
\lim_{i\rightarrow+\infty}(K_i\cap C^-(t_k))=K\cap C^-(t_k),\ \mbox{for each}\ k.
\end{equation*}
Fixing $i$, we choose a compact set $\beta\subset\Omega_{C^\circ}$ with $\omega_i\subset\text{int}\,\beta$, and choose $k$ so large that $\widetilde{\partial}K^{(\beta)}\subset C^-(t_k)$. Then by Lemma 13 in \cite{Schneider-A_weighted_Minkowski_theorem} we have $\boldsymbol{\nu}^*_K(\omega_i)=\boldsymbol{\nu}^*_{K^{(\beta)}}(\omega_i)$,
and by Lemma 2 in \cite{Schneider-A_weighted_Minkowski_theorem} one also has
\begin{equation*}
\lim_{i\rightarrow\infty}(K_i^{(\beta)}\cap C^-(t_k))=K^{(\beta)}\cap C^-(t_k)).
\end{equation*}
Combining Lemma \ref{Lemma-weak-continuity-Gauss-surface-area-measure} with Lemma \ref{Lemma-continuity-Gauss-co-volume}, we know $S_{\gamma_n}(K_i,\cdot)\llcorner\omega_i$ converges to $S_{\gamma_n}(K^{(\beta)},\cdot)\llcorner\omega_i$ weakly and $\overline{\gamma}_n(K_i)$ converges to $\overline{\gamma}_n(K)$. Thus, we have
\begin{equation*}
\overline{\gamma}_n(K_i)^{\alpha-1}S_{\gamma_n}(K_i,\cdot)\llcorner\omega_i\rightarrow
\overline{\gamma}_n(K)^{\alpha-1}S_{\gamma_n}(K^{(\beta)},\cdot)\llcorner\omega_i\ \ \mbox{weakly}.
\end{equation*}
Since $\overline{\gamma}_n(K_i)^{\alpha-1}S^\Theta_{n-1}(K_i,\cdot)\llcorner\omega_i=\mu\llcorner\omega_i$, we have
\begin{equation*}
\begin{aligned}
\overline{\gamma}_n(K)^{\alpha-1}S_{\gamma_n}(K,\omega')&=\overline{\gamma}_n(K)^{\alpha-1}S_{\gamma_n}(K^{(\beta)},\omega')\\
&=\lim_{i\rightarrow+\infty}\overline{\gamma}_n(K_i)^{\alpha-1}S_{\gamma_n}(K_i,\omega')\\
&=\lim_{i\rightarrow+\infty}\mu_i(\omega')\\
&=\lim_{i\rightarrow+\infty}\mu(\omega')\\
&=\mu(\omega')
\end{aligned}
\end{equation*}
for each Borel set $\omega'\subset\omega_i$. Hence, by $\bigcup_{i=1}^{+\infty}\omega_i=\Omega_{C^\circ}$ we have $\overline{\gamma}_n(K)^{\alpha-1}S_{\gamma_n}(K,\cdot)=\mu$.\\
$(ii)$ Conversely, it is clear by Lemma \ref{Lemma-finiteness-Gauss-surface-area-measure}.
\end{proof}

\end{document}